\documentclass[11pt]{article}
\usepackage{amsmath,amsfonts,amssymb,amsthm}
\usepackage{mathrsfs}
\usepackage{latexsym}
\usepackage[english]{babel}

\usepackage{graphicx}
\usepackage[usenames,dvipsnames]{xcolor}

\renewenvironment{proof}[1][\proofname]{{\bfseries #1.} }{\qed}
\def\Cov{{\rm Cov\,}}

\newcommand{\field}[1]{\mathbb{#1}}

\newcommand{\R}{\field{R}}

\newcommand{\Var}{{\rm Var}}

\newcommand{\e}{{\rm e}}

\newcommand{\eps}{\varepsilon}

\def\authors#1{{ \begin{center} #1 \vspace{0pt} \end{center} } \smallskip}
\def\institution#1{{\sl \begin{center} #1 \vspace{0pt} \end{center} } }
\def\inst#1{\unskip $^{#1}$}
\def\title#1{{\huge\bf  \begin{center} #1 \vspace{0pt} \end{center}  } \smallskip}

\def\E{{\mathbb{ E}}}

\def\P{{\mathbb{P}}}

\def\paref#1{(\ref{#1})}

\newtheorem{theorem}{Theorem}[section]
\newtheorem{example}[theorem]{Example}

\newtheorem{lemma}[theorem]{Lemma}

\newtheorem{definition}[theorem]{Definition}

\newtheorem{remark}[theorem]{Remark}
\numberwithin{equation}{section}

\begin{document}

\date{Jun 2020}

\title{\sc Moderate Deviation estimates for\\ Nodal Lengths of\\ Random Spherical Harmonics}
\authors{\large Claudio Macci\inst{*}, Maurizia Rossi\inst{\diamond,}\footnote{Corresponding author. E-mail address: \texttt{maurizia.rossi@unimib.it}} and Anna Paola Todino\inst{\star}}
\institution{\inst{*}Dipartimento di Matematica, Universit\`a di Roma ``Tor Vergata"\\
\inst{\diamond,}Dipartimento di Matematica e Applicazioni, Universit\`a di Milano-Bicocca\\
\inst{\star}Fakult\"at f\"ur Mathematik, Ruhr-Universit\"at Bochum}

\begin{abstract}

We prove Moderate Deviation estimates for nodal lengths of random spherical harmonics 
both on the whole sphere and on shrinking spherical domains. Central Limit Theorems for the latter were recently established in Marinucci, Rossi and Wigman (2020) and Todino (2020), respectively. Our proofs are based on the combination of a Moderate Deviation Principle by Schulte and Th\"ale (2016) for sequences of random variables living in a fixed Wiener chaos with a well-known
result based on the concept of exponential equivalence. 
 
\smallskip


\smallskip


\end{abstract}

\section{Introduction: background and motivations}

\subsection{Random spherical harmonics}\label{sub:RSH}

Let $\mathbb S^2$ denote the two-dimensional unit sphere with the round metric and $\Delta$ the spherical Laplacian. In standard spherical coordinates $(\theta, \varphi)\in [0,\pi]\times [0,2\pi)$, where $\theta$ is the colatitude, the metric takes the form 
\begin{equation}\label{metric}
dx^2 = d\theta^2 + \sin^2\theta d\varphi^2.
\end{equation}
It is well-known that the spectrum of $\Delta$ is purely discrete, its eigenvalues are of the form $-\ell(\ell+1)$ where $\ell\in \mathbb N$ and, for each $\ell$, the family of the so-called spherical harmonics of degree $\ell$ $\lbrace Y_{\ell,m}, m = 1,\dots, 2\ell +1\rbrace$ is a real orthonormal basis of the $\ell$-th eigenspace \cite[\S 3.4]{MPbook}. 
\begin{definition}\label{defRSH}
For $\ell\in \mathbb N$, the $\ell$-th random spherical harmonic $T_\ell$ is a centered Gaussian field on $\mathbb S^2$ whose covariance kernel is given by 
\begin{equation}\label{defCov} 
\Cov(T_\ell(x), T_\ell(y)) = P_\ell(\cos d(x,y)),\quad x,y\in \mathbb S^2,
\end{equation} 
where $P_\ell$ denotes the Legendre polynomial \cite[\S 13.1.2]{MPbook} of degree $\ell$ and $d(x,y)$ the spherical geodesic distance (see (\ref{metric})) between $x$ and $y$. 
\end{definition}
Equivalently, one can define $T_\ell$ as follows
\begin{equation}\label{defT} 
T_\ell(x) := \sqrt{\frac{4\pi}{2\ell+1} } \sum_{m=1}^{2\ell+1} a_{\ell,m} Y_{\ell,m}(x),\quad x\in \mathbb S^2,
\end{equation} 
where $\{a_{\ell,m}, m=1,\dots, 2\ell+1\}$ are standard Gaussian and independent random variables, defined on a probability space $(\Omega, \mathcal F, \mathbb P)$. Actually, from the addition formula \cite[(3.42)]{MPbook} for random spherical harmonics the covariance kernel of $T_\ell$ in (\ref{defT}) is given by  (\ref{defCov}). It is immediate that $T_\ell$ is isotropic, and that it is $\mathbb P-$a.s. an eigenfunction of the spherical Laplacian with eigenvalue $-\ell(\ell+1)$. We can assume $\lbrace T_\ell, \ell \in \mathbb N\rbrace$ to be independent random fields, indeed they are the Fourier components of isotropic Gaussian fields on the sphere, see e.g. \cite{BMV07} and \cite[\S 5, \S6]{MPbook}.  

Now we recall the Hilb's asymptotic formula: let $\epsilon >0$, uniformly for $\theta \in [0, \pi-\epsilon]$ and $\ell\in \mathbb N_{\ge 1}$
\begin{equation}\label{hilb}
P_\ell(\cos \theta) = \sqrt{\frac{\theta}{\sin \theta}}\, J_0 \left ( \left (\ell+1/2 \right ) \theta \right ) + O \left (\ell^{-3/2}\right ),
\end{equation} 
where $J_0$ is the Bessel function \cite[\S 1.71]{Sze75} of the first kind of order zero, see the conventions below for the meaning of the $O-$notation. The scaling limit, as $\ell\to +\infty$, of $T_\ell$ is the so-called Berry's Random Wave model which, according to Berry's conjecture \cite{Ber77}, should model the local behavior of high-energy deterministic eigenfunctions on ``generic chaotic" surfaces. 

\smallskip

\noindent \emph{Conventions.}  Given two sequences of positive real numbers $\lbrace x_n, n\in \mathbb N \rbrace$ and $\lbrace y_n, n\in \mathbb N\rbrace$ we will write $x_n=O(y_n)$ if the ratio $x_n/y_n$ is asymptotically bounded, and $x_n = o(y_n)$ if $\lim_{n\to +\infty} x_n/y_n=0$. Moreover, we will write $x_n \approx y_n$ if both $ x_n = O(y_n)$ and $y_n = O(x_n)$ hold.

\subsection{Nodal lengths: asymptotic distribution}\label{subsecNodal}

Let us consider the nodal set 
$
T_\ell^{-1}(0) := \lbrace x\in \mathbb S^2 : T_\ell(x) = 0\rbrace.
$
It is well-known that $T_\ell^{-1}(0)$ is an $\mathbb P-$a.s. smooth curve whose connected components are homeomorphic to the circle. 
We are interested in the high-energy geometry of the nodal set, in particular in the asymptotic behavior, as $\ell\to +\infty$, of the nodal length 
\begin{equation}\label{nodalDef}
\mathcal L_\ell:=\text{length}(T_\ell^{-1}(0)).
\end{equation}
  The latter received great attention also in view of Yau's conjecture on nodal volumes of deterministic eigenfunctions on compact Riemannian manifolds \cite{Yau82}. We collect the main known results on the distribution of $\mathcal L_\ell$ in a single theorem (Theorem \ref{thS} below): the expected length was studied in \cite{Ber85}, the asymptotic variance in \cite{Wig10} and the second order fluctuations of $\mathcal L_\ell$ in \cite{MRW20}. In order to state them, let us recall that, for $X,Y$ integrable real random variables, the Wasserstein distance between $X$ and $Y$ is defined as $d_W(X,Y) := \sup_{h} | \mathbb E[h(X)] - \mathbb E[h(Y)] |$, where the supremum is taken over the set of Lipschitz functions whose Lipschitz constant is $\leq 1$. From now on, $Z\sim \mathcal N(0,1)$ will denote a standard Gaussian random variable.
\begin{theorem}\label{thS}
For every $\ell \in \mathbb N$ 
 \begin{equation}\label{mean1}
\mathbb E[\mathcal L_\ell] = \frac{4\pi}{2\sqrt 2} \sqrt{\ell(\ell+1)}.
\end{equation}
As $\ell\to +\infty$, we have 
\begin{equation}\label{eqVar}
\Var(\mathcal L_\ell) =  \frac{1}{32} \log \ell \, (1+ o(1)).
\end{equation} 
Moreover, denoting 
 $$
\widetilde{\mathcal L}_\ell := \frac{\mathcal L_\ell - \mathbb E[\mathcal L_\ell]}{\sqrt{\Var(\mathcal L_\ell)}},
$$
a quantitative CLT in Wasserstein distance holds, i.e., as $\ell\to +\infty$
 \begin{equation}\label{Wass1}
 d_W \left ( \widetilde{\mathcal L}_\ell, Z       \right ) = O \left ( (\log \ell)^{-1/2} \right ).
 \end{equation}
\end{theorem}  
Note that, from (\ref{mean1}) and (\ref{eqVar}) we have $\mathcal L_\ell/ \sqrt{\ell(\ell+1)}\to 4\pi/(2\sqrt{2})$ $\mathbb P-$a.s. as $\ell\to +\infty$, consistently with Yau's conjecture \cite{Yau82}. From (\ref{Wass1}) we conclude in particular that the asymptotic distribution of the nodal length is Gaussian. Moreover, from (\ref{Wass1}) and \cite[(C.2.6)]{NP12} 
\begin{equation}
d_{Kol}\left ( \widetilde{\mathcal L}_\ell, Z       \right ) = O \left ( (\log \ell)^{-1/4} \right ),\end{equation}
where for arbitrary real random variables $X,Y$, the Kolmogorov distance between $X$ and $Y$ is defined as $d_{Kol}(X,Y) := \sup_{x\in \mathbb R} \left | \mathbb P(X\le x) - \mathbb P(Y\le x)\right |$.  

\subsubsection{Shrinking spherical domains}

In \cite{Tod18} the asymptotic behavior of the nodal length in shrinking spherical domains was investigated. 
Fixed a point $x_0\in \mathbb S^2$, let $B_r$ denote the spherical cap of radius $r>0$ centered at $x_0$ and consider the length of nodal lines in $B_r$ 
$$
\mathcal L_{\ell, B_r} := \text{length} (T_\ell^{-1}(0) \cap B_r).
$$
We recall that the area of $B_{r}$ is equal to $2\pi(1-\cos r)$. We have the following results summarized in the next theorem.
\begin{theorem}\label{thB}
For any $r>0$
\begin{equation}\label{mean2}
\mathbb E[\mathcal L_{\ell, B_r}] = \frac{2\pi(1-\cos r)}{2\sqrt{2}} \sqrt{\ell(\ell+1)}.
\end{equation} 
For a sequence of radii $\lbrace r_\ell, \ell \in \mathbb N\rbrace$ converging to zero not too rapidly ($r_\ell \ell \to +\infty$) we have, as $\ell\to +\infty$,
\begin{equation}\label{var2}
\Var(\mathcal L_{\ell, B_{r_\ell}}) = \frac{1}{256} r_\ell^2 \log (r_\ell  \ell) + O (r_\ell^2).
 \end{equation} 
Moreover, denoting 
$$\widetilde{\mathcal L}_{\ell, B_{r_{\ell}}}:= \frac{\mathcal L_{\ell, B_{r_\ell}} - \mathbb E[\mathcal L_{\ell, B_{r_\ell}}]}{ \sqrt{\Var(\mathcal L_{\ell, B_{r_\ell}})}}$$
 we have a quantitative CLT in Wasserstein distance as $\ell\to +\infty$
\begin{equation}\label{Wass2}
d_W \left ( \widetilde{\mathcal L}_{\ell, B_{r_{\ell}}}, Z \right ) = O \left (  \left(\log r_\ell \ell \right)^{-1/2}  \right ).
\end{equation}
\end{theorem}
Indeed (\ref{Wass2}) follows from the proof of Theorem 2.2 in \cite{Tod18} (at the end of \S 5.2.3) where
in particular it is proved that the distribution of the nodal length in shrinking domains is asymptotically Gaussian. As in the previous case, we can also deduce that, as $\ell \to +\infty$,
\begin{equation}
d_{Kol}\left ( \widetilde{\mathcal L}_{\ell, B_{r_{\ell}}}, Z       \right ) = O  \left (  \left(\log r_\ell \ell \right)^{-1/4}  \right ).\end{equation}

In this paper we are interested in refinements of Central Limit Theorems stated above (Theorems \ref{thS} and \ref{thB}). 

\subsection{Moderate Deviation Principles}\label{sub:MDPs}

The theory of Large Deviations allows an asymptotic computation of small
probabilities at exponential scales. Here we start by recalling a basic definition.
From now on let $(\Omega, \mathcal F, \mathbb P)$ be a complete probability space,
and $\lbrace X_n, n\in \mathbb N\rbrace$ a sequence of real-valued random variables: for each $n$ the map
$$
X_n : \Omega \to \mathbb R$$
is measurable with respect to $\mathcal F$ and $\mathcal B(\mathbb R)$ (the Borel $\sigma$-field of $\mathbb R$) 
on $\Omega$ and $\mathbb R$ respectively. 
\begin{definition}\label{defLDP}  We say that $\lbrace X_n, n\in \mathbb N\rbrace$ satisfies the Large Deviation Principle (LDP) with speed $0\le s_n\nearrow +\infty$ and good rate function\footnote{See \S 1.2 in \cite{DZ98} for details.} $\mathcal I$ if for every $\alpha\geq 0$ the level set $\lbrace x : \mathcal I(x) \le \alpha \rbrace$ is compact and for all $B\in \mathcal B(\mathbb R)$ we have 
\begin{equation*}
- \inf_{x\in \mathring{B}} \mathcal I(x)\le \liminf_{n\to +\infty} \frac{1}{s_n} \log \mathbb P(X_n\in B) \le \limsup_{n\to +\infty} \frac{1}{s_n} \log \mathbb P(X_n\in B) \le - \inf_{x\in \bar{B}} \mathcal I(x),
\end{equation*}
where $\mathring B$ (resp. $\bar B$) denotes the interior (resp. the closure) of $B$. 
\end{definition}

Let us now recall the notion of exponential equivalence \cite[Definition 4.2.10]{DZ98} related to the question whether the LDP for $\lbrace Y_n, n\in \mathbb N\rbrace$ can be deduced from the LDP for $\lbrace X_n, n\in \mathbb N\rbrace$, $\lbrace Y_n, n\in \mathbb N\rbrace$ being another sequence of real-valued random variables.
\begin{definition}\label{expdef} We say that $\lbrace X_n, n\in \mathbb N\rbrace$ and $\lbrace Y_n, n\in \mathbb N\rbrace$ are exponentially equivalent at speed $0\le s_n\nearrow +\infty$ if, for every $\delta>0$,
\begin{equation}\label{exp_equiv}
\limsup_{n\to +\infty} \frac{1}{s_n} \log \mathbb P ( | X_n - Y_n | > \delta ) = -\infty.
\end{equation} 
\end{definition}
As far as the LDP is concerned, exponentially equivalent sequences of random variables are indistinguishable \cite[Theorem 4.2.13]{DZ98}. 
\begin{lemma}\label{thDZ}
Assume that $\lbrace X_n, n\in \mathbb N\rbrace$ satisfies the LDP with speed $s_n$ and good rate function $\mathcal I$.
Then, if $\lbrace X_n, n\in \mathbb N\rbrace$ and $\lbrace Y_n, n\in \mathbb N\rbrace$ are exponentially equivalent at
speed $s_n$, the same LDP holds for $\lbrace Y_n, n\in \mathbb N\rbrace$.
\end{lemma} 
A Moderate Deviation Principle (MDP) is a class of LDPs for families of random variables depending on the choice of certain scalings in a suitable class. Moreover, all these LDPs (whose speed function depends on the scaling) are ruled by the same quadratic rate function vanishing at zero (actually, usually one deals with families of centered random variables, or asymptotically centered). 
In several cases the choice of the scaling parameters allows to fill the gap between the convergence in probability to a constant and the convergence in law to a centered Gaussian random variable. 
\begin{example}\label{exMDP}
Let $\lbrace Z_n, n\in \mathbb N\rbrace$ be a sequence of i.i.d. real-valued random variables such that the moment-generating function 
\begin{equation*}
\mathbb R \ni \lambda\mapsto   \log \mathbb E\left [\e^{ \lambda Z_1} \right ] \end{equation*}
is finite in some ball around the origin, $\mathbb E[Z_1]=0$ and $\Var(Z_1) = 1$. Let us define for each $n$ the random variable
\begin{equation*}
X_n :=  \sum_{i=1}^n Z_i/\sqrt{n}.
\end{equation*}
For any sequence of positive numbers $\lbrace a_n, n\in \mathbb N\rbrace$ such that 
$$
a_n\to +\infty,\qquad a_n/\sqrt n \to 0,
$$
a MDP with speed $a_n^2$ and rate function $\mathcal I(x) := x^2/2, x\in \mathbb R$, holds for $\lbrace X_n, n\in \mathbb N\rbrace$, namely for every Borel set $B\subset \mathbb R$ 
\begin{equation*}
- \inf_{x\in \mathring{B}} \mathcal I(x)\le \liminf_{n\to +\infty} \frac{1}{a^2_n} \log \mathbb P(X_n/a_n\in B) \le \limsup_{n\to +\infty} \frac{1}{a^2_n} \log \mathbb P(X_n/a_n\in B) \le - \inf_{x\in \bar{B}} \mathcal I(x),
\end{equation*}
where $\mathring B$ (resp. $\bar B$) denotes the interior (resp. the closure) of $B$. 
\end{example} 
The classical example \cite[Theorem 3.7.1]{DZ98} recalled just above concerns the empirical means of i.i.d. random variables; indeed, if these random variables have finite moment generating function in a neighborhood of the origin, a class of LDPs holds filling the gap between the asymptotic regimes of the Law of Large Numbers (LLN) and the Central Limit Theorem (CLT): indeed, using the same notation as in Example \ref{exMDP}, as $n\to +\infty$ we have
\begin{equation*}
X_n/\sqrt{n} \to 0\quad \mathbb P-a.s.,\qquad X_n \mathop{\to}^d Z\sim \mathcal N(0,1),
\end{equation*}
where $\mathop{\to}^d$ denotes convergence in distribution. 

 Furthermore, for completeness, we recall that a LDP linked to a LLN is provided by the celebrated Cram\'er Theorem \cite[Theorem 2.2.3]{DZ98}.

\section{Main results and outline of the paper}

\subsection{Statement of main results}

Our main results concern MDPs which refine Theorems \ref{thS} and \ref{thB}, namely a class of LDPs for nodal lengths of random spherical harmonics for certain scalings $\{a_\ell,\ell\in\mathbb{N}\}$ that tend to infinity 
slowly (see conditions (\ref{cond1}) and (\ref{cond2}) in Theorems \ref{th1} and \ref{th2} below), with speed 
$a_\ell^2$, and common quadratic rate function $\mathcal I(x)=x^2/2$, $x\in \mathbb R$. Moreover, for 
$a_\ell=1$ (and in such a case the condition $a_\ell \to +\infty$ in (\ref{cond1}) and (\ref{cond2}) fails) we have the convergence in law to the standard Normal distribution (Theorem \ref{thS} and Theorem \ref{thB}).
Recall the preliminaries
in Section \ref{sub:MDPs} and the discussion just after the statement of Lemma \ref{thDZ}.
\begin{theorem}\label{th1}
Let $\lbrace a_\ell, \ell \in \mathbb N\rbrace$ be any sequence of positive numbers such that, as $\ell\to +\infty$,
\begin{equation}\label{cond1}
a_\ell \to +\infty,\qquad \qquad a_\ell/\sqrt{\log \log \ell} \to 0.
\end{equation}
The sequence of random variables 
$$\left \lbrace \widetilde{\mathcal L}_\ell/a_\ell, \ell \in \mathbb N \right \rbrace$$ satisfies a MDP with speed $a_\ell^2$ and rate function $\mathcal I(x) = x^2/ 2$, $x\in \mathbb R$.
\end{theorem} 
From now on,  $\lbrace r_\ell, \ell \in \mathbb N\rbrace$ is a sequence of radii such that (see Theorem \ref{thB})
$$r_\ell \ell \to +\infty.
$$ 
\begin{theorem}\label{th2}
Let $\lbrace a_\ell, \ell \in \mathbb N\rbrace$ be any sequence of positive numbers such that, as $\ell\to +\infty$,
\begin{equation}\label{cond2}
a_\ell \to +\infty,\qquad \qquad a_\ell/\sqrt{\log \log r_\ell \ell} \to 0.
\end{equation}
The sequence of random variables 
$$\left \lbrace \widetilde{\mathcal L}_{\ell, B_{r_\ell}}/a_\ell, \ell \in \mathbb N \right \rbrace$$ satisfies a MDP with speed $a_\ell^2$ and rate function $\mathcal I(x) = x^2/ 2$, $x\in \mathbb R$.
\end{theorem} 
To the best of our knowledge, Theorem \ref{th1} and Theorem \ref{th2} are the first MD estimates for Lipschitz-Killing curvatures of excursion sets of Laplacian Gaussian eigenfunctions on manifolds and on shrinking domains on manifolds, respectively. 
\begin{remark}[Lipschitz-Killing curvatures]\rm
The conditions $a_\ell/\sqrt{\log \log \ell}\to 0$ in (\ref{cond1}) and $a_\ell/ \sqrt{\log \log r_\ell \ell}\to 0$ in (\ref{cond2}) are plausibly not optimal; this is a drawback of the proof technique we decided to adopt. Our choice was based on the fact that we do not have any information on the moment-generating function of the nodal length and, more importantly, on the shortness of our argument as well as on its flexibility. Indeed, our strategy can be immediately adapted to prove MD estimates for the other Lipschitz-Killing (LK) curvatures of excursion sets at any level of random spherical harmonics such as the length of level curves, the excursion area and the Euler-Poincar\'e characteristic (see e.g. \cite{CM18, MW14, Tod19} and the references therein). Actually, our approach works well independently of the underlying manifold as long as the random model is Gaussian and long-memory, for instance in the case of arithmetic random waves, see \cite{BMW18, Cam19, MPRW16, NPR19, PR18} and the references therein.
\end{remark}

\subsection{On the proofs of the main results}

The nodal length $\mathcal L_\ell$ in (\ref{nodalDef}) is a finite-variance functional of the Gaussian field $T_\ell$ (Definition \ref{defRSH}) hence 
it can be written as an orthogonal series, the so-called \emph{Wiener-\^Ito chaos expansion}, converging in $L^2(\mathbb P)$ of the form 
\begin{equation}\label{exp1}
\mathcal L_\ell = \mathbb E[\mathcal L_\ell] + \sum_{q=1}^{+\infty} \mathcal L_\ell[q],
\end{equation}
where $\mathcal L_\ell[q]$ is the orthogonal projection of $\mathcal L_\ell$ onto the so-called Wiener chaos of order $q$. The expansion in (\ref{exp1}) relies on the fact that the family of suitably normalized Hermite polynomials $\lbrace H_q, q\in \mathbb N\rbrace$
is an orthonormal basis for the space of square integrable functions on the real line with respect to the Gaussian density. Recall that 
$H_0 \equiv 1$ and
 \begin{equation}\label{hermite}
 H_{q}(t) := (-1)^q \phi^{-1}(t) \frac{d^q}{dt^q} \phi(t), \qquad t\in \mathbb R, q\in \mathbb N_{\ge 1},
 \end{equation}
 where $\phi$ denotes the standard Gaussian density. In particular $H_1(t)=1$, $H_2(t)=t^2-1$, $H_3(t)= t^3 -3t$ and $H_4(t) = 3t^4 - 6t^2 +3$. 
It turns out that
 $\mathcal L_\ell - \mathbb E[\mathcal L_\ell]$ is asymptotically equivalent, in the $L^2(\mathbb P)$-sense, to its fourth chaotic component $\mathcal L_\ell[4]$
which is moreover fully correlated, as $\ell\to +\infty$, to the \emph{sample trispectrum} of $T_\ell$. More precisely
let us define, for $\ell\in \mathbb N$,
\begin{equation}\label{M}
\mathcal M_\ell := -\frac14 \sqrt{\frac{\ell(\ell+1)}{2}} \frac{1}{4!} \int_{\mathbb S^2} H_4(T_\ell(x))\,dx;
\end{equation}
Theorem 1.2 in \cite{MRW20} states that, as $\ell\to +\infty$,
\begin{equation}\label{L2equivalence}
\mathbb E \left [ \left | \widetilde{\mathcal L}_\ell -  \widetilde{\mathcal M}_\ell \right |^2 \right ] = O\left ((\log \ell)^{-1}\right ),
\end{equation}
where 
$$\widetilde{\mathcal M}_\ell := \frac{\mathcal M_\ell}{\sqrt{\Var(\mathcal M_\ell)}}.$$
Now recall that the sequence of random variables $\lbrace \widetilde{\mathcal M}_\ell, \ell \in \mathbb N\rbrace$ lives in the fourth Wiener chaos. On that space, convergence in law to a standard Gaussian random variable can be proved  \cite{NP05, NP09} showing that its fourth cumulant goes to zero, hence
\begin{equation}\label{4cum}
\text{Cum}_4 ( \widetilde{\mathcal M}_\ell  )\to 0.
\end{equation}
Theorem 1 in \cite{ScTh16} ensures that under (\ref{4cum}) a MDP holds for the normalized sample trispectrum. See also \cite{DE13} for MDP via cumulants and \cite{Led90} for LDP for Wiener chaos. 

In this paper, as a preliminary result (Lemma \ref{lem1}), firstly we will establish a MDP for $\lbrace \widetilde{\mathcal M}_\ell/a_\ell, \ell \in \mathbb N\rbrace$ with the rate function $\mathcal I$ in Theorem \ref{thS} and speed $a_\ell^2$, whenever as $\ell\to +\infty$
\begin{equation*}\label{a1}
a_\ell \to +\infty,\qquad \qquad a_\ell/(\log\ell)^{1/7}\to 0.
\end{equation*}
Then, in order to deduce a MDP for the whole series on the right hand side of (\ref{exp1}), i.e. to establish Theorem \ref{th1}, we will prove that $\lbrace \widetilde{\mathcal M}_\ell/a_\ell, \ell\in \mathbb N\rbrace$ and $\lbrace \widetilde{\mathcal L}_\ell/a_\ell, \ell \in \mathbb N	\rbrace$ are exponentially equivalent (Definition \ref{expdef}) at speed $a_\ell^2$ provided that $a_\ell$ goes to infinity sufficiently slowly (see (\ref{cond1})) according to (\ref{L2equivalence}). Finally Lemma \ref{thDZ} will allow to conclude. 

The proof of Theorem \ref{th2} relies on the same ideas as those developed just before for the proof of Theorem \ref{th1}. 

\subsection{Plan of the paper}

In \S \ref{secWiener} we recall both the notion of Wiener chaos and the chaotic expansion (\ref{exp1}) for the nodal length of random spherical harmonics. In particular, in \S \ref{sec_red} we recall the reduction principle for nodal lengths on the sphere and on shrinking domains leading in particular to (\ref{L2equivalence}). Finally in \S \ref{sec_proofs} we give the proofs of our main results. 

\section{Nodal lengths and Wiener chaos}\label{secWiener}

\subsection{Wiener chaos}\label{sub_wiener}

It is well-known \cite[Proposition 1.4.2]{NP12} that the family $ \{H_q/\sqrt{q!}, q\in \mathbb N\}$ of suitably normalized Hermite polynomails (\ref{hermite}) is a complete orthonormal system in the space of square integrable real functions $L^2(\phi)$ with respect to the standard Gaussian measure on the real line.

Random spherical harmonics \eqref{defT} are linear combinations of i.i.d. standard Gaussian random variables $\{a_{\ell,m} : \ell = 1, 2, \dots, m=1, \dots, 2\ell+1\}$; we define accordingly the space $X$ to be the closure in $L^2(\mathbb{P})$ of $\text{lin}\{a_{\ell,m} : \ell = 1, 2, \dots, m=1, \dots, 2\ell+1\}$,
thus $X$ is a real centered Gaussian Hilbert subspace
of $L^2(\mathbb{P})$.
Now let $q\in \mathbb N$; the $q$-th Wiener chaos $C_q$ associated with $X$ is defined as the closure in $L^2(\mathbb{P})$ of all real finite linear combinations of random variables of the form
$$
H_{p_1}(x_1) H_{p_2}(x_2)\cdots H_{p_k}(x_k)
$$
for $k\in \mathbb N_{\ge 1}$, where $p_1,...,p_k \in \mathbb N$ satisfy $p_1+\cdots+p_k = q$, and $(x_1, x_2, \dots, x_k)$ is a standard Gaussian vector extracted
from $X$ (plainly, $C_0 = \mathbb{R}$ and $C_1 = X$). Note that (from (\ref{defT})) for every $\ell$ the random fields $T_{\ell}$ and $\nabla T_\ell$ viewed as collections of
Gaussian random variables
indexed by $x\in \mathbb S^2$ are all lying in $X$.

Taking into account the orthonormality and completeness of $\{H_q/\sqrt{q!}, q\in \mathbb N\}$ in $L^2(\phi)$, together with a monotone class argument (see e.g. \cite[Theorem 2.2.4]{NP12}), one can prove that $C_q \,\bot\, C_{q'}$ in  $L^2(\mathbb{P})$ whenever $q\neq q'$, and moreover
\begin{equation*}
L^2_X(\mathbb{P}) = \bigoplus_{q=0}^\infty C_q,
\end{equation*}
where $L^2_X(\mathbb P) := L^2(\Omega, \sigma(X), \mathbb P)$, that is, every finite-variance real-valued functional $F$ of $X$ admits a unique representation as a series, converging in $L^2_X(\mathbb P)$, of the form
\begin{equation}\label{e:chaos2}
F = \sum_{q=0}^\infty F[q],
\end{equation}
$F[q]:=\text{proj}(F \, | \, C_q)$ being the orthogonal projection of $F$ onto $C_q$ (in particular, $F[0]= \E [F]$).
 For a complete discussion on Wiener chaos see \cite[\S 2.2]{NP12} and the references therein. 

\subsection{Nodal length: chaos expansion}

Let $B\subseteq \mathbb S^2$ be a ``nice" subset of the sphere. For our purpose it suffices to take $B$ as the whole sphere or a spherical cap. The nodal length $\mathcal L_{\ell, B} := \text{length}(T_\ell^{-1}(0) \cap B)$ in $B$ at frequency $\ell$ (plainly, $\mathcal L_{\ell, \mathbb S^2}  \equiv \mathcal L_\ell$) 
can be formally written as 
\begin{equation}\label{rapSphere}
\mathcal L_{\ell, B} = \int_{B} \delta_0(T_\ell(x)) \|\nabla T_\ell(x)\|\,dx,
\end{equation}
where $\delta_0$ stands for the Dirac mass in $0$, $\nabla T_\ell$ is the gradient field and $\| \cdot \|$ denotes the Euclidean norm in $\R^2$. Indeed, let us consider the $\eps$-approximating random variable 
$$
\mathcal L_{\ell,B}^\eps := \frac{1}{2\eps} \int_{B} \chi_{[-\eps, \eps]}(T_\ell(x))\, \| \nabla T_\ell(x)\|\, dx,
$$
where $0 < \eps\ll 1$ and $\chi_{[-\eps, \eps]}$ denotes the indicator function of the interval $[-\eps, \eps]$. It is possible to prove that 
\begin{equation*}
\lim_{\eps \to 0} \mathcal L_{\ell,B}^\eps = \text{length}(T_\ell^{-1}(0)\cap B)
\end{equation*}
both $\mathbb P-$a.s. and in $L^2(\P)$, see \cite{MRW20, NPR19}, thus justifying \paref{rapSphere}. 
In particular, $\mathcal L_{\ell,B}\in L^2_X(\mathbb P)$.
The integral representation \paref{rapSphere} can be equivalently written as 
\begin{equation}\label{rapSphere2}
\mathcal L_{\ell, B} 
=\sqrt{\frac{\ell(\ell+1)}{2}}\int_{B} \delta_0(T_\ell(x)) \|\widetilde \nabla T_\ell(x)\|\,dx,
\end{equation}
where  $\widetilde \nabla$ is the normalized gradient, i.e. $\widetilde \nabla:= \nabla / \sqrt{\frac{2}{\ell(\ell +1)}}$ (thus pointwise the components of the normalized gradient have unit variance, see \S 3.2.1 in \cite{MRW20} for details). 
Let us now recall the chaotic expansion \paref{e:chaos2} for $\mathcal L_{\ell, B}$ 
\begin{equation}\label{chaos_decomp}
\mathcal L_{\ell,B} = \sum_{q=0}^{+\infty} \mathcal L_{\ell,B}[2q],
\end{equation}
where $\mathcal L_{\ell,B}[2q]$ denotes the orthogonal projection of $\mathcal L_{\ell,B}$ onto  $C_{2q}$. (Note that projections on odd chaoses vanish since the integrand functions in \paref{rapSphere2} are both even.)
In \cite[$\S2$]{MRW20} the terms of the series on the right hand side of \paref{chaos_decomp} 
are explicitly given (see also \cite{MPRW16}). 
Let us introduce the two sequences of real numbers $\lbrace \beta_{2k}\rbrace_{k=0}^{+\infty}$ and $\lbrace \alpha_{2n,2m}\rbrace_{n,m=0}^{+\infty}$ corresponding to the (formal) chaotic coefficients of the Dirac mass at $0$ and the Euclidean norm respectively: for $k, n,m\in \mathbb N$ 
\begin{equation*}
\beta_{2k} := \frac{1}{\sqrt{2\pi}} H_{2k}(0),\qquad \qquad \alpha _{2n,2m}:=\sqrt{\frac{\pi }{2}}\frac{(2n)!(2m)!}{n!m!}\frac{1}{2^{n+m}%
}p_{n+m}\left(\frac{1}{4}\right),
\end{equation*}
where $p_{N}$ is the swinging factorial coefficient
$
p_{N}(x):=\sum_{j=0}^{N}(-1)^{N+j}\left(
\begin{array}{c}
N \\
j%
\end{array}%
\right) \frac{(2j+1)!}{(j!)^{2}}x^{j} 
$, $x\in \mathbb R$.
The $2q$-th chaotic projection of the nodal length restricted to $B$ is
\begin{equation}\label{chaos_decomp_sphere}
\begin{split}
 \mathcal L_{\ell,B} [2q] = &\sqrt{\frac{\ell (\ell
+1)}{2}} \sum_{u=0}^{q}\sum_{k=0}^{u}\frac{\alpha
_{2k,2u-2k}\beta _{2q-2u}}{(2k)!(2u-2k)!(2q-2u)!} \\
&\times \int_{B}H_{2q-2u}(T_{\ell }(x))H_{2k}(\widetilde \partial
_{1;x}T_{\ell }(x))H_{2u-2k}(\widetilde \partial
_{2;x}T_{\ell }(x))\,dx,
\end{split}
\end{equation}
where we use spherical coordinates (colatitude $\theta ,$ longitude $%
\varphi $) and for $x=(\theta _{x},\varphi _{x})$ we are using the notation
\begin{equation*}
\widetilde \partial _{1;x}=\left. (\ell(\ell+1)/2)^{-1/2}\cdot \frac{\partial }{\partial \theta }\right\vert
_{\theta =\theta _{x}},\qquad \widetilde \partial _{2;x}= (\ell(\ell+1)/2)^{-1/2}\cdot\left. \frac{1}{\sin \theta }%
\frac{\partial }{\partial \varphi }\right\vert _{\theta =\theta _{x},\varphi
=\varphi _{x}}.
\end{equation*}
Obviously $\mathcal L_{\ell,B}[0] = \frac{\text{area}(B)}{2\sqrt{2}} \sqrt{\ell(\ell+1)}  = \mathbb E[\mathcal L_{\ell,B}]$. 

\subsection{Nodal lengths: reduction principles}\label{sec_red}

\subsubsection{On the sphere}\label{sub_sphere}

From (\ref{chaos_decomp_sphere}) for $q=1$, an application of Green's formula yields  (see \cite{MR19} and the references therein) 
\begin{equation}
\mathcal L_\ell[2] = 0.
\end{equation}
Let us consider, as in (\ref{M}), the sample trispectrum 
\begin{equation*}
\mathcal M_\ell := -\frac14 \sqrt{\frac{\ell(\ell+1)}{2}}\frac{1}{4!} \int_{\mathbb S^2} H_4(T_\ell(x))\,dx;
\end{equation*}
from the properties of Hermite polynomials recalled in \S \ref{sub_wiener} and by (\ref{hilb}) we have
\begin{equation}\label{eqVarM}
\mathbb E[\mathcal M_\ell] = 0, \qquad \qquad
\Var(\mathcal M_\ell) = \frac{1}{32} \log \ell + O(1), \text{ as } \ell\to +\infty
\end{equation} (see \cite[Lemma 3.2]{MW14}), cf. (\ref{eqVar}). Moreover, in \cite{MRW20} it has been shown that the fourth chaotic projection $\mathcal L_\ell[4]$ (see (\ref{chaos_decomp_sphere}) for $q=2$) is asymptotically (as $\ell\to +\infty$) equivalent, in the $L^2(\mathbb P)$-sense, to $\mathcal M_\ell$, and that the tail $\sum_{q\ge 3} \mathcal L_\ell[2q]$ of the chaotic series (\ref{chaos_decomp})  is negligible. To be more precise, 
\begin{equation}\label{red1}
\mathbb E\left [ \left | \widetilde{\mathcal L}_\ell -\widetilde{\mathcal M}_\ell   \right |^2 \right ] = O\left ( (\log \ell)^{-1}\right ),\quad \text{ as } \ell\to+\infty,
\end{equation}
which is (\ref{L2equivalence}). An application of the Fourth Moment Theorem \cite{NP05, NP09} gives \cite{MW14}
\begin{equation*}
d_W \left ( \widetilde{\mathcal M}_\ell, Z\right ) \le \sqrt{\text{Cum}_4 (\widetilde{\mathcal M}_\ell  )} = O \left ( (\log \ell)^{-1} \right ),\qquad \text{ as } \ell\to +\infty,
\end{equation*}
that together with the estimate (\ref{red1}) proves (\ref{Wass1}). 

\subsubsection{On shrinking spherical domains} \label{shrinking}

Let us define the local sample trispectrum as 
\begin{equation}\label{Mloc}
\mathcal M_{\ell,B_{r_\ell}} := -\frac14 \sqrt{\frac{\ell(\ell+1)}{2}}\frac{1}{4!} \int_{B_{r_\ell}} H_4(T_\ell(x))\,dx.
\end{equation}
It has zero mean and variance (recall that $r_\ell \ell\to +\infty$) given by
\begin{equation}\label{eqVarL}
\Var(\mathcal M_{\ell,B_{r_\ell}}) = \frac{1}{256}r_\ell^2 \log r_\ell \ell + O(r_\ell^2)
\end{equation}
 (see \cite[Proposition 3.4]{Tod18}).
In \cite[Proposition 3.3]{Tod18} the asymptotic full correlation between the local nodal length $\mathcal{L}_{\ell,B_{r_\ell}}$ and the local sample trispectrum has been established and, in view of the orthogonality of the projections, this entails that $\mathcal{M}_{\ell,B_{r_\ell}}$ is the leading term of the chaos expansion of $\mathcal{L}_{\ell,B_{r_\ell}}$ in (\ref{chaos_decomp_sphere}). Indeed, all the other projections are proved to be $O(r_\ell^2)$ in the $L^2-$sense (see the supplement article to \cite{Tod18}, Appendix C) and hence we can conclude that 
\begin{equation}\label{L2palla}
\mathbb E\left [ \left | \widetilde{\mathcal L}_{\ell,B_{r_\ell}} -\widetilde{\mathcal M}_{\ell,B_{r_\ell}}   \right |^2 \right ] = O\left ( (\log r_\ell \ell)^{-1}\right ).
\end{equation}
Hence the Fourth Moment Theorem \cite{NP05, NP09} gives \cite[Lemma 5.4]{Tod18}
\begin{equation*}
d_W \left ( \widetilde{\mathcal M}_{\ell, B_{r_\ell}}, Z\right )  \le \sqrt{\text{Cum}_4 (\widetilde{\mathcal M}_{\ell, B_{r_\ell}}  )} = O \left ( (\log r_\ell \ell)^{-1/2} \right ) 
\end{equation*}
that together with the estimate (\ref{L2palla}) proves (\ref{Wass2}). 

\section{Proofs of the main results}\label{sec_proofs}

\subsection{Proof of Theorem \ref{th1}}

Bearing in mind \S \ref{sub_sphere} we start with an auxiliary lemma of independent interest which provides the MDP for the sample trispectrum on the sphere. 

\begin{lemma}\label{lem1}
Let $\lbrace a_\ell, \ell\in \mathbb N\rbrace$ be any sequence of positive numbers such that, as $\ell\to +\infty$,
\begin{equation}\label{condM}
a_\ell \to +\infty,\qquad \qquad a_\ell/(\log\ell)^{1/7}\to 0.
\end{equation}
Then the sequence of random variables $\lbrace \widetilde{\mathcal M}_\ell/ a_\ell, \ell\in \mathbb N\rbrace$ satisfies a MDP with speed $a_\ell^2$ and rate function $\mathcal I(x) =x^2/2$, $x\in \mathbb R$.
\end{lemma}
\begin{proof}
The random variable $\mathcal M_\ell$ belongs to the fourth Wiener chaos $C_4$ for each $\ell\in \mathbb N$. 
Recalling (\ref{eqVarM}),
from \cite[Lemma 3.3]{MW14}, we have that, as $\ell\to +\infty$, $\text{Cum}_4(\mathcal M_\ell) \approx 1$. 
Hence we have, as $\ell\to +\infty$, 
\begin{equation}\label{cum}
\text{Cum}_4(\widetilde{\mathcal M}_\ell) = \frac{\text{Cum}_4(\mathcal M_\ell)}{\Var(\mathcal M_\ell)^2} \approx \frac{1}{\log^2\ell}.
\end{equation}
Now let $a_\ell$ be a positive sequence such that, as $\ell \to +\infty$,
\begin{equation}\label{cumM}
a_\ell \to +\infty,\qquad \qquad \qquad a_\ell/ (\Delta_\ell)^{1/3} \to 0,
\end{equation}
where 
\begin{equation}\label{Delta}
\Delta_\ell := \left ( \sqrt{\text{Cum}_4  (\widetilde{\mathcal M}_\ell )}\right )^{-3/7} \approx (\log \ell)^{3/7}.
\end{equation}
Note that (\ref{cumM}) is equivalent to (\ref{condM}).
Corollary 2 in \cite{ScTh16} ensures that the sequence $\widetilde{\mathcal M}_\ell/a_\ell$ satisfies a MDP with speed $a_\ell^2$ and rate function $\mathcal I(x) = x^2/ 2$, $x\in \mathbb R$.

\end{proof}
\begin{remark}\label{remCP}\rm 
Our argument leads to further results, namely expansions \`a la Cram\'er-Petrov \cite{RSS78} for the sample trispectrum. We state here the result in a simplified form taken from Theorem 5 (i) in \cite{ScTh16}: there exist universal constants $c_0, c_1, c_2 >0$ such that for $\Delta_\ell \ge c_0$, $0\le x \le c_1 \Delta_\ell^{1/3}$ ($\Delta_\ell$ is defined as in (\ref{Delta})) it holds that 
\begin{equation*}
\left | \log \frac{\mathbb P(\widetilde{\mathcal M}_\ell \ge x)}{1 - \Phi(x)} \right | \le c_2 \frac{1+x^3}{\Delta_\ell^{1/3}},\qquad \left |  \log \frac{\mathbb P(\widetilde{\mathcal M}_\ell \le -x)}{ \Phi(-x)} \right |  \le c_2 \frac{1+x^3}{\Delta_\ell^{1/3}},
\end{equation*}
where $\Phi$ denotes the standard Gaussian cumulative distribution function. An anonymous referee raised the interesting question whether expansions \`a la Cram\'er-Petrov hold for nodal lengths on the sphere and in shrinking spherical domains. We leave this as a topic for future research.
\end{remark}
\begin{proof}[Proof of Theorem \ref{th1}]
We want to combine Theorem \ref{thDZ} and Lemma \ref{lem1} for every speed $a_\ell$ satisfying (\ref{cond1}); note that $a_\ell/\sqrt{\log \log \ell} \to 0$ implies $a_\ell/(\log \ell)^{1/7}\to 0$ (see (\ref{condM})). In order to do so, we have to check that for every $\delta >0$
\begin{equation}\label{limsup}
\limsup_{\ell\to +\infty}\frac{1}{a_\ell^{2}} \log \mathbb P( a_\ell^{-1} |  \widetilde{\mathcal {L}}_\ell -  \widetilde{\mathcal M}_\ell | > \delta) = -\infty,
\end{equation}
i.e., that the sequences of random variables $\lbrace  \widetilde{\mathcal L}_\ell /a_\ell, \ell\in \mathbb N\rbrace$ and $\lbrace  \widetilde{\mathcal {M}}_\ell/a_\ell, \ell\in \mathbb N\rbrace$ are exponentially equivalent (Definition \ref{expdef}) at speed $a_\ell^2$. 
Thanks to Markov inequality we have 
\begin{equation*}
\begin{split}
\mathbb P( a_\ell^{-1} |  \widetilde{\mathcal L}_\ell  -  \widetilde{\mathcal M}_\ell | > \delta) \le \frac{a_\ell^{-2} \mathbb E[| \widetilde{\mathcal L}_\ell  -  \widetilde{\mathcal M}_\ell |^2]}{\delta^2}
\end{split}
\end{equation*} 
so that 
\begin{equation}\label{app1}
\begin{split}
\limsup_{\ell\to +\infty}a_\ell^{-2} \log \mathbb P( a_\ell^{-1} |   \widetilde{\mathcal L}_\ell -  \widetilde{\mathcal M}_\ell | > \delta) &\le \limsup_{\ell\to +\infty}a_\ell^{-2} \log \frac{a_\ell^{-2} \mathbb E[| \widetilde{\mathcal L}_\ell  -  \widetilde{\mathcal M}_\ell |^2]}{\delta^2}\cr
&= \limsup_{\ell\to +\infty}a_\ell^{-2} \log  \mathbb E[| \widetilde{\mathcal L}_\ell  -  \widetilde{\mathcal M}_\ell |^2].
\end{split}
\end{equation} 
Plugging (\ref{L2equivalence}) into (\ref{app1}) we get (\ref{limsup}) whenever $a_\ell = o(\sqrt{\log \log \ell})$ as in (\ref{cond1}).

\end{proof}

\subsection{Proof of Theorem \ref{th2}}
Here we refer to \S \ref{shrinking} and we follow the same lines of the proof of Theorem \ref{th1}. We first prove the following lemma which is of independent interest.

\begin{lemma}\label{lem1bis}
	Let $\lbrace a_\ell, \ell\in \mathbb N\rbrace$ be any sequence of positive numbers such that, as $\ell\to +\infty$,
	\begin{equation}\label{condAP}
	a_\ell \to +\infty,\qquad \qquad a_\ell/(\log r_\ell \ell)^{1/14} \to 0.
	\end{equation}
	Then the sequence of random variables $\lbrace \widetilde{\mathcal M}_{\ell, B_{r_\ell}}/ a_\ell, \ell\in \mathbb N\rbrace$ satisfies a MDP with speed $a_\ell^2$ and rate function $\mathcal I(x)=x^2/2$, $x\in \mathbb R$.
\end{lemma}
\begin{proof}
	The random variable $\mathcal{M}_{\ell, {B_{r_\ell}}}$ in (\ref{Mloc}) belongs to the fourth Wiener chaos for each $\ell\in \mathbb N$. Recalling (\ref{eqVarL}) and that from \cite[Lemma 5.4]{Tod18}, as $\ell\to +\infty$, $\text{Cum}_4(\mathcal{M}_{\ell, B_{r_\ell}}) =O( r_\ell^4 \log r_\ell \ell)$,
 we have 
	\begin{equation}\label{cumbis}
	\text{Cum}_4(\widetilde{\mathcal M}_{ \ell, B_{r_\ell}}) = \frac{\text{Cum}_4(\mathcal{M}_{\ell, B_{r_\ell}})}{\Var(\mathcal{M}_{\ell, B_{r_\ell}})^2} =O\left ( \frac{1}{\log r_\ell \ell} \right ).
	\end{equation}
	Now let $a_\ell$ be a positive sequence such that, as $\ell \to +\infty$,
	\begin{equation}\label{cumAP}
	a_\ell \to +\infty,\qquad \qquad \qquad a_\ell/ (\Delta_{\ell, r_\ell})^{1/3} \to 0,
	\end{equation}
	where 
	$$
	\Delta^{-1}_{\ell, r_\ell} := \left ( \sqrt{\text{Cum}_4 (\widetilde{\mathcal M}_{ \ell, B_{r_\ell}}  )}\right )^{3/7} = O\left (  (\log r_\ell \ell)^{-3/14} \right),\ \text{ as } \ell\to +\infty.
	$$
	Note that (\ref{condAP}) implies (\ref{cumAP}). It follows that,
as in the proof of Lemma \ref{lem1}, Corollary 2 in \cite{ScTh16} ensures that the sequence $\lbrace \widetilde{\mathcal M}_{\ell,B_{r_\ell}}/a_\ell, \ell \in \mathbb N\rbrace$ satisfies a MDP with speed $a_\ell^2$ and rate function $\mathcal I(x) = x^2/ 2$.

\end{proof}

Analogous results as those in Remark \ref{remCP} hold for the sample trispectrum restricted to a shrinking ball. 

\begin{proof}[Proof of Theorem \ref{th2}] Along the same lines of the proof of Theorem \ref{th1} we want to combine Theorem \ref{thDZ} and Lemma \ref{lem1bis} for every speed $a_\ell$ satisfying (\ref{cond2}); note that $a_\ell/\sqrt{\log \log r_\ell \ell} \to 0$ implies $a_\ell/(\log r_\ell \ell)^{1/14}\to 0$ (see (\ref{condAP})). To this aim, we have to check that for every $\delta >0$
	\begin{equation}\label{limsupbis}
	\limsup_{\ell\to +\infty}\frac{1}{a_\ell^{2}} \log \mathbb P( a_\ell^{-1} |  \widetilde{\mathcal{L}}_{\ell,B_{r_\ell}} -  \widetilde{\mathcal M}_{\ell,B_{r_\ell}} | > \delta) = -\infty. 
	\end{equation}
Applying Markov inequality we have 
	\begin{equation*}
	\begin{split}
	\mathbb P( a_\ell^{-1} |  \widetilde{\mathcal{L}}_{\ell,B_{r_\ell}} -  \widetilde{\mathcal M}_{\ell,B_{r_\ell}} | > \delta) \le \frac{a_\ell^{-2} \mathbb E[| \widetilde{\mathcal {L}}_{\ell,B_{r_\ell}} -  \widetilde{\mathcal M}_{\ell,B_{r_\ell}} |^2]}{\delta^2}
	\end{split}
	\end{equation*} 
and hence
	\begin{equation}\label{app4}
	\begin{split}
	\limsup_{\ell\to +\infty}a_\ell^{-2} \log \mathbb P( a_\ell^{-1} |   \widetilde{\mathcal {L}}_{\ell,B_{r_\ell}} -  \widetilde{\mathcal M}_{\ell,B_{r_\ell}} | > \delta) &\le \limsup_{\ell\to +\infty}a_\ell^{-2} \log \frac{a_\ell^{-2} \mathbb E[| \widetilde{\mathcal {L}}_{\ell,B_{r_\ell}} -  \widetilde{\mathcal M}_{\ell,B_{r_\ell}} |^2]}{\delta^2}\cr
	&= \limsup_{\ell\to +\infty}a_\ell^{-2} \log  \mathbb E[| \widetilde{\mathcal {L}}_{\ell,B_{r_\ell}} -  \widetilde{\mathcal M}_{\ell,B_{r_\ell}}|^2].
	\end{split}
	\end{equation} 
	Using (\ref{L2palla}) in (\ref{app4}) and taking $a_\ell = o(\sqrt{\log \log r_\ell	\ell})$, (\ref{limsupbis}) holds. 
	
\end{proof}

\subsection*{Acknowledgements} 

The authors would like to thank an anonymous referee for his/her valuable comments which improved the quality of the present work.
C.M. has been supported by MIUR Excellence Department Project awarded to the Department of Mathematics, 
University of Rome Tor Vergata (CUP E83C18000100006) and by GNAMPA-INdAM (project: 
\emph{Stime asintotiche: principi di invarianza e grandi deviazioni}). M.R. has been supported by 
GNAMPA-INdAM (project: \emph{Propriet\`a analitiche e geometriche di campi aleatori}) and the 
ANR-17-CE40-0008 project \emph{Unirandom}. A.P.T. has been funded by the German Research 
Foundation (DFG) via RTG 2131 and by GNAMPA-INdAM (project: 
\emph{Stime asintotiche: principi di invarianza e grandi deviazioni}).

\end{document}